\documentclass[11pt,reqno,palatino]{amsart}
\usepackage{mypack}

\graphicspath{{pics/}}

\title[On periodic orbits]{}
\author[J.B. van den Berg,
F. Pasquotto, T.O. Rot and R.C.A.M. Vandervorst]{}

\begin{document}
\maketitle \noindent {\huge {\bf On periodic orbits in cotangent bundles of non-compact manifolds.}\footnote{Jan Bouwe van den Berg supported by NWO Vici grant 639.033.109, Federica Pasquotto supported in part by NWO Meervoud grant 632.002.901, and
Thomas Rot supported by
NWO grant 613.001.001.}} \vskip.8cm
\begin{sloppypar}
\noindent
Jan Bouwe van den Berg (janbouwe.vanden.berg@vu.nl),
Federica Pasquotto (f.pasquotto@vu.nl), Thomas O.~Rot (t.o.rot@vu.nl) and Robert C.A.M.~Vandervorst (r.c.a.m.vander.vorst@vu.nl)
\vskip.3cm

\noindent {\it Department of Mathematics, VU University Amsterdam, De Boelelaan 1081a,
1081 HV Amsterdam, the Netherlands.}
\vskip.3cm

\noindent {\bf Abstract.}  This paper is concerned with the existence of periodic orbits on energy hypersurfaces in cotangent bundles of Riemannian manifolds defined by mechanical Hamiltonians. In \cite{bpv} it was proved that, provided  certain geometric assumptions are satisfied, regular mechanical hypersurfaces in $\mR^{2n}$, in particular non-compact ones, contain periodic orbits if one homology group among the top half does not vanish.  In the present paper we extend the above mentioned existence result to a class of hypersurfaces in cotangent bundles of Riemannian manifolds with flat ends.

\vspace{0.1in}

\noindent {\em AMS Subject Class:} 37J05, 37J45, 70H12.

\noindent{\em Keywords:} Periodic orbits, Weinstein conjecture, Hamiltonian dynamics, free loop space, linking sets.

\section{Introduction}

The question of existence of periodic orbits of a Hamiltonian vector field $X_H$ on a given regular energy level, i.e. a level set $\Sigma=H^{-1}(0)$ of the Hamiltonian function $H$, with $dH\not=0$ on $\Sigma$, has been a central question in Hamiltonian dynamics and symplectic topology which has generated some of the most interesting recent developments in those areas. The existence of a periodic orbit does not depend on the Hamiltonian itself, but only on the geometry of the energy level that the Hamiltonian defines. For this reason one also speaks of closed characteristics of the energy level. 

After the first pioneering existence results of Rabinowitz\ \cite{rabinowitzstar} and\ \cite{rabinowitzenergy} and Weinstein~\cite{weinsteinconvex} for star-shaped and convex hypersurfaces respectively, Viterbo~\cite{viterboweinstein} proved the existence of closed characteristics on all compact hypersurfaces of $\mR^{2n}$ of so called \emph{contact type}. The latter notion was introduced by Weinstein as a generalization of both convex and star-shaped~\cite{weinsteinhypothesis}.  These first results were obtained by variational methods applied to a suitable (indefinite) action functional.

More recently, Floer, Hofer, Wysocki~\cite{floerhoferwysocki} and Viterbo~\cite{viterbofunctors}, provided an alternative proof of the same results (and much more) using the powerful tools of symplectic homology or Floer homology for manifolds with boundary.

Up to now though, very little is known about periodic orbits on \emph{non-compact} energy hypersurfaces: even the Floer homology type of technique mentioned above breaks down when one drops the compactness assumptions.  It is clear that additional geometric and topological assumptions are needed in order to make up for the lack of compactness. In~\cite{bpv} we were able to formulate a set of such assumptions that led to an existence result for the case of \emph{mechanical hypersurfaces} in $\mR^{2n}$, that is, hypersurfaces arising as level sets of Hamiltonian functions of the form kinetic plus potential energy. 

Mechanical hypersurfaces in cotangent bundles are an important class of contact manifolds since they occur naturally in conservative mechanical dynamics. In the case of compact mechanical hypersurfaces Bolotin~\cite{bolotin}, Benci~\cite{benci}, and Gluck and Ziller~\cite{Gluck:1983ut} show the existence of a closed characteristic on $\Sigma$ via closed geodesics of the Jacobi metric on the configuration manifold. 
A more general existence result for cotangent bundles is proved by Hofer and Viterbo in \cite{hoferviterbo} and improved in \cite{viterboexact}: Any connected compact hypersurface of contact type over a simply connected manifold has a closed characteristic, which confirms the  Weinstein Conjecture in cotangent bundles of simply connected manifolds.  However, the existence of closed characteristics for non-compact mechanical hypersurfaces is not covered by the result of Hofer and Viterbo and fails without additional geometric conditions. In this paper we address the question for non-compact mechanical hypersurfaces, in cotangent bundles of non-compact (smooth) Riemannian manifolds $(M,g)$. 

\subsection{Main result}

A Riemannian manifold is said to have \emph{flat ends} if the curvature tensor vanishes outside a compact set. The main theorem of this paper is the following existence result.
\begin{theorem}
\label{thm:main2}
Let $H:T^*M\rightarrow \mR$ be the Hamiltonian $H(q,\theta_q)=\frac 12g_q^*(\theta_q,\theta_q)+V(q)$, where
\begin{itemize}
\item $(M,g)$ is an $n$-dimensional complete orientable Riemannian manifold with flat ends;
\item $\Sigma=H^{-1}(0)$ is a regular hypersurface, i.e.,$dH\not=0$ on $\Sigma$;
\item $V$ is \emph{asymptotically regular}, i.e. there exist a compact set $K$ and a constant $\asympconstant>0$ such that
$$
|\grad V(q)| \geq \asympconstant, ~~\text{for}~~ q\in M\setminus K
  ~~\text{and}~~ 
  \frac{\|\hes V(q)\|}{|\grad V(q)|} \to 0,
~~\text{as}~~ d(q,K)\rightarrow \infty. 
$$
\end{itemize}
Assume moreover that there exists an integer $0\leq k\leq n-1$ such that 
\begin{enumerate}
\item[(i)]  $H_{k+1}(\Lambda M)=0$ and $H_{k+2}(\Lambda M)=0$, and
\item[(ii)] $H_{k+n}(\Sigma)\neq 0$.
\end{enumerate}
Then $\Sigma$ has a periodic orbit which is contractible in $T^*M$.
\end{theorem}

Here $\Lambda M$ denotes the free loop space of $H^1$ loops into $M$. The proof of Theorem~\ref{thm:main2} follows the scheme of the proof of the existence result for non-compact hypersurfaces in $\mR^{2n}$ presented in \cite{bpv}, and when arguments are identical as that paper, we refer to it without proof. We regard periodic orbits as critical points of a suitable action functional $\A$. The functional does not satisfy the Palais-Smale condition. Therefore we introduce a sequence of approximating functionals $\A_\epsilon$, for $\epsilon>0$, which do satisfy the Palais-Smale condition. Critical points of $\A_\epsilon$ satisfying certain bounds converge to critical points of $\A$ as $\epsilon\rightarrow 0$. Next, based on the assumptions on the topology of $\Sigma$ and $M$, we construct linking sets in $M$ and lift these to linking sets in the free loop space, where we apply a linking argument to produce critical points for the approximating functionals satisfying the appropriate bounds. These critical points then converge to a critical point of $\A$ as $\epsilon\rightarrow 0$. Because we construct the linking sets in the component of the loop space containing the contractible loops, this critical point corresponds to a contractible loop. Hence, due to our method of proof, in this paper we do not find non-contractible loops. 

One of the main issues in cotangent bundles (as opposed to $\mR^{2n}$) is that the linking arguments get more involved due to the topology of $M$. Another difficulty is that curvature terms appear in the analysis of the functional, which require some care. 

Theorem~\ref{thm:main2} directly generalizes the results of~\cite{bpv}. In~\cite{bpv} examples are given that show that both topological and geometric assumptions on $\Sigma$ are necessary. Theorem\ \ref{thm:main2} also improves the result in the $\mR^{2n}$ case, as it requires weaker assumptions on the metric.

After completion of this research, and based on the results presented in this paper, Suhr and Zehmisch~\cite{suhrzehmisch} were able to prove that some of the technical hypotheses of Theorem~\ref{thm:main2} can be removed. In particular, the main result holds for cotangent bundles over manifolds of bounded geometry and no assumptions on the homology of the loop space are necessary.

\subsection{Acknowledgement}
We would like to thank the referee for carefully reading our manuscript and for the many constructive comments and suggestions, which substantially improved the presentation of our result.

\section{The Palais-Smale condition}
\label{sec:PS}
Periodic orbits on $\Sigma$ can be regarded as critical points of the action functional 
\begin{equation*}
\B(q,T)=\int_0^T\left\{\frac{1}{2}|q'(t)|^2-V(q(t))\right\}\,dt,
\end{equation*}
for mappings $q: [0,T] \to M$ with $q(0)=q(T)$. Via the coordinate transformation
\begin{equation*}
(q(t),T)\mapsto (c(s),\tau)=\bigl(q(sT),\log(T)\bigr).
\end{equation*}
we obtain the rescaled action functional 
\begin{equation*}
\A(c,\tau)=\frac{e^{-\tau}}{2}\int_0^1|c'(s)|^2\,ds-e^\tau\int_0^1
V(c(s))\,ds,
\end{equation*}
for mappings $c: \mS \to M$ and $\tau \in \mR$, where $\S^1=[0,1]/\{0,1\}$ is the parameterized circle. The natural domain of the functional $\A$ is $\Lambda M\times \mR$, where $\Lambda M$ denotes the space of loops of Sobolev regularity $H^1$ in $M$. It is convenient to define $\cE(c)=\frac 12 \int|c'(s)|^2ds$ and $\cW(x)=\int_0^1 V(c(s))ds$ such that $\A(c,\tau)=e^{-\tau} \cE(c)-e^\tau \cW(c)$. There are some basic inequalities for the various metrics on the loop space which will be used in the analysis. The proofs, as well as more details on the construction of the loop space and the metrics, can be found in the books of Klingenberg \cite{Klingenberg, Klingenberg_Riemannian}. Let $c,\tilde c\in \Lambda M$ and $s,\tilde s\in \mS^1$, then
\begin{eqnarray}
d_M(c(s), c(\tilde s)) &\le& \sqrt{|s-\tilde s|}\sqrt{2\cE(c)},\label{eqn;metric-1}\\
d_{C^0}(c,\tilde c) &\le& \sqrt{2}\,\, d_{\Lambda M}(c,\tilde c),\label{eqn:metric-2}
\end{eqnarray}
where $d_{\Lambda M}$ is the metric induced by the Riemannian metric on $\Lambda M$, and $d_{C^0}(c,\tilde c)=\sup_{s\in \mS^1}d_M(c(s),\tilde c(s))$. This metric is complete if the metric induced by the Riemannian metric on $M$ is complete. See for example \cite[Theorem 2.4.7]{Klingenberg_Riemannian} where the proof also shows that $\Lambda M$ is complete as a metric space if $M$ is complete as a metric space. For $\xi\in T_c\Lambda M$, we have the estimate
\begin{equation}
\label{eq:inequality}
\Vert\xi\Vert_{L^2}\leq\Vert\xi\Vert_{C_0}\leq\sqrt{2}\Vert\xi\Vert_{H^1}.
\end{equation}
A direct computation gives the variation of the action. 
\begin{lemma}
\label{lem:1stvar}
The action $\A: \Lambda M \times \mR \to \mR$ is continuously differentiable. For any $(\xi,\sigma) \in T_{(c,\sigma)}\,\Lambda M \times \mR$ the first variation is given by
\begin{align*}
\label{eqn:1stvar}
d\A(c,\tau) (\xi,\sigma) & = e^{-\tau} \int_0^1 \langle c'(s),\triangledown \xi(s)\rangle ds
- e^\tau \int_0^1 \langle \grad V(c(s)), \xi(s)\rangle ds\\
&\quad~~~~~~ - \int_0^1 \Bigl[ \frac{e^{-\tau}}{2} |c'(s)|^2 + e^\tau 
V(c(s)) \Bigr] \sigma\,ds\\
&= d_c \A(c,\tau) \xi - \Bigl( e^{-\tau} \cE(c) + e^\tau \cW(c)\Bigr) \sigma,
\end{align*}
where 
  the gradient $\grad V$ is taken with respect to the metric $g$ on $M$.
\end{lemma}

The functional $\A$ does not satisfy the Palais-Smale condition. We therefore approximate this functional by functionals $\A_\epsilon$, and show that the approximating functionals do satisfy PS. We then find critical points of the approximating functionals using a linking argument. Finally we show that these critical points converge to a critical point of $\A$ as $\epsilon\rightarrow 0$. The approximating, or penalized, functionals are defined by

\begin{equation*}
\A_\epsilon(c,\tau)=\A(c,\tau)+\epsilon(e^{-\tau}+e^{\tau/2}).
\end{equation*}

The term $\epsilon e^{-\tau}$ penalizes orbits with short periods, and $\epsilon e^{\tau/2}$ penalizes orbits with long periods. Recall that,. for $\epsilon>0$ fixed,  a sequence $\{(c_n, \tau_n)\} \in \Lambda M\times\mR$ is called  a \emph{Palais-Smale sequence} for 
$\A_\epsilon$, if:
\begin{itemize}
\item[(i)]    there exist constants $a_1,a_2>0$ such that $a_1\leq\A_\epsilon(c_n, \tau_n)\leq a_2$;
\item[(ii)] $\Vert d \A_\epsilon(c_n, \tau_n)\Vert \to 0$ as
$n$ tends to $\infty$. 
\end{itemize}

The metric in Condition (ii) is the dual Riemannian metric on $T^*\Lambda M\times \mR$. This can be equivalently rewritten as 
\begin{equation}
\label{e:almostcrit}
d\A_\epsilon(c_n, \tau_n)(\xi, \sigma)=
\langle \grad\A_\epsilon(c_n, \tau_n), (\xi,\sigma)\rangle_{H^1\times\mR}=
o(1)(\|\xi\|_{H^1}+|\sigma|),
\end{equation}
as $n\to \infty$ and $(\xi,\sigma) \in T_{(c_n,\tau_n)}\Lambda M\times\mR$. Condition (i) implies that, by passing to a subsequence if necessary, 
$\A_\epsilon(c_n, \tau_n)\to a_\epsilon$, with $a_1\leq a_\epsilon\leq a_2$. In what follows we tacitly assume we have passed to such a subsequence. 
\begin{remark}
We will only consider Palais-Smale sequences that are positive, thus $a_1>0$. The functionals $\A_\epsilon$ satisfy the Palais-Smale condition for critical levels $a_\epsilon>a_1>0$.
\end{remark}
The relation between $\A$ and $\A_\epsilon$ gives:
$$
d\A_\epsilon (c_n, \tau_n)(\xi, \sigma)= d \A(c_n, \tau_n)(\xi, \sigma) - \epsilon \Bigl( e^{-\tau_n} - \frac{1}{2} e^{\tau_n/2}\Bigr)\sigma.
$$
Proceeding as in Lemma 7 of~\cite{bpv} we get the following estimates for Palais Smale sequences. 
\begin{align}{2}
 & e^{-\tau_n} \cE(c_n) + \epsilon \bigl(2 e^{-\tau_n} +\frac{1}{2}
 e^{\tau_n/2}\bigr) =a_\epsilon + o(1),
  \label{e:es1}\\
  & e^{\tau_n} \cW(c_n) - \epsilon\frac{3}{4}e^{\tau_n/2} 
   = -\frac{a_\epsilon}{2} +o(1),\qquad \hbox{as~~~~} n \to\infty.
  \label{e:es1b}
\end{align}
From these estimates, we get a priory bounds on $\tau_n$ as in Lemma 8 of~\cite{bpv} and combining this with estimate~\bref{e:es1} for the kinetic energy, we also get a bound on the kinetic energy.

\begin{lemma}

\label{lem:est}
Let $(c_n,\tau_n)$ be a Palais-Smale sequence. There are constants $T_0 < T_1$ and $C$ (depending on $\epsilon$) such that $T_0 \le \tau_n\le T_1$ and $\Vert c'_n\Vert_{L^2}^2 = 2\cE(c_n) \le C$.
\end{lemma}

The following proposition establishes the Palais-Smale condition for the action $\A_\epsilon$, with $\epsilon>0$.
\begin{proposition}
\label{lem:PS}
Let $(c_n,\tau_n)$ be a Palais-Smale sequence for $\A_\epsilon$, $\epsilon>0$. Then
$(c_n,\tau_n)$ has an accumulation point $(\ce,\te) \in \Lambda M \times \mR$ that is a critical point, i.e. 
 $d\A_\epsilon(\ce, \te)=0$ and the action is bounded $0<a_1 \le \A_\epsilon(\ce,\te) = a_\epsilon \le a_2$.
\end{proposition}

\begin{proof}
From Lemma~\ref{lem:est}
 we have that $\cE(c_n) \le C$ and $|\tau_n|\le C'$,
with the constants $C,C'>0$ depending only on $\epsilon$. Fix $s_0 \in \S^1$, then by Eq.\ (\ref{eqn;metric-1}) we have
$d_M(c_n(s),c_n(s_0)) \le \sqrt{|s-s_0|} \sqrt{2C} \le \sqrt{2C}$, and therefore $c_n(s) \in B_{\sqrt{2C}}(c_n(s_0))$, for all $s\in \S^1$ and all  $n$. Since $K\subset M$ is compact, its diameter is finite. It follows, that if $d_M(c_n(s_0),K) \to \infty$ that  there exists an $N$ such that $c_n(s) \in M\setminus K$, for all $s\in \S^1$ and all $n\ge N$. The argument that it is impossible that $d_M(c_n(s_0),K)\rightarrow \infty$ follows from asymptotic regularity and is identical to the proof of Lemma~9 of~\cite{bpv}.

Thus $d_M(c_n(s_0),K) \le C''$ and therefore there exists an $0<R<\infty$ such that $c_n(s) \in B_R(K)$ for all $s\in \S^1$ and all $n\ge N$. Since $(M,g)$ is complete, the Hopf-Rinow Theorem implies that $B_R(K)$ is compact and thus $\{c_n(s)\} \subset M$ is pre-compact for any fixed $s\in \S^1$. The sequence $\{c_n(s)\}$ is point wise relatively compact and  equicontinuous by Eq.\  (\ref{eqn;metric-1}). Therefore, by the generalized version of the Arzela-Ascoli Theorem \cite{munkres} there exists a subsequence $c_{n_k}$ converging in $C^0(\S^1,M)$ (uniformly) to a continuous limit $c_\epsilon \in C^0(\S^1,M)$. It remains to show that $c_\epsilon$ is an accumulation point in $\Lambda M$, thus in $H^1$ sense.

Due to the above convergence in $C^0(\S^1,M)$, the sequence $\{c_n\}$ can be assumed to be contained  in a
fixed chart $\bigl(\cU(c_0),\exp_{c_0}^{-1}\bigr)$, for a fixed $c_0 \in C^\infty(\S^1,M)$. Following \cite{Klingenberg}
it suffices to show that $\exp_{c_0}^{-1} c_n$ is a Cauchy sequence in $T_{c_0}\Lambda M = H^1(c_0^*TM)$.
This final technical argument is identical to Theorem 1.4.7 in \cite{Klingenberg}, which proves that $\{c_n\}$ has an accumulation point in $(c_\epsilon,\tau_\epsilon) \in\Lambda M\times \mR$, proving the Palais-Smale condition for $\A_\epsilon$. The limit points satisfy 
$d\A_\epsilon(c_\epsilon,\tau_\epsilon) = 0$, and $\A_\epsilon(c_\epsilon,\tau_\epsilon)=a_\epsilon$.
 \end{proof}

For critical points of $\A_\epsilon$ we prove additional  a priori estimates on   $\tau_\epsilon$.
The latter imply a priori estimates on $c_\epsilon$.
This allows us to pass to the limit as $\epsilon \to 0$.  

We start with pointing out that critical points of the penalized action $\A_\epsilon$ satisfy the following
Hamiltonian identity
\begin{equation}
\label{eqn:ham-1}
\frac{e^{-2\te}}{2} |\ce'(s)|^2 +  V(\ce(s))\equiv  \epsilon\Bigl( -e^{-2\te} +  \frac{1}{2} e^{-\te/2}\Bigr) = \tilde \epsilon.
\end{equation}
Thus the critical point $(\ce,\te)$ corresponds to a periodic orbits on $\Sigma_{\tilde\epsilon}$. Via the transformation $q_\epsilon(t) = \ce(t e^{-\tau})$ and the Legendre transform of $(q_\epsilon,q'_\epsilon)$ to a  curve $\gamma_\epsilon$ on the cotangent bundle we see that the Hamiltonian action is
$$
\A^H_\epsilon(\gamma_\epsilon,\te) = \int_{\gamma_\epsilon} \Lambda + \epsilon\bigl(e^{-\te} + e^{\te/2}\bigr),
$$
where $\Lambda$ is the tautological 1-form on $T^*M$.

A regular hypersurface of a mechanical Hamiltonian, wether it is compact or not, is always is of contact type, cf.~\cite{Rot:ww}. 
In the case that $\Sigma$ is defined for a Hamiltonian with $V$ asymptotically regular, an explicit contact form can be constructed and a stronger contact type condition holds. Consider the vector field:
\begin{equation}
   v(q) =- \frac{\grad V(q)}{1+\left|\grad V(q)\right|^2},
\label{e:choice2}
\end{equation}
and the function $f:T^*M \to \mR$ defined by $f(x) = \theta_q(v(q))$, for all $x=(q,\theta_q) \in T^*M$. For $\kappa>0$, define the 1-form 
$\Theta = \Lambda + \kappa df$ . Clearly, $d\Theta = \Omega$, the standard symplectic form on the cotangent bundle. Define the energy surfaces $\Sigma_\epsilon = \{x\in T^*M~|~H(x) = \epsilon\}$.

\begin{proposition}
\label{prop:second-contact}
Let $V$ be asymptotically regular. Then there exists  $\epsilon_0,\kappa_0>0$, such that for every $-\epsilon_0<\epsilon< \epsilon_0$, $\Theta = \Lambda + \kappa df$ restricts to a contact form on $\Sigma_\epsilon$, for all $0<\kappa\le \kappa_0$. Moreover, for every $\kappa$, there exists a constant $a_\kappa>0$ such that
$$
\Theta(X_H) \ge a_\kappa >0,\quad\hbox{for all}~~x\in \Sigma_\epsilon\quad\hbox{and for all}\quad~-\epsilon_0<\epsilon
< \epsilon_0.
$$ 
The energy surfaces $\Sigma_\epsilon$ are said to be of uniform contact type. 
\end{proposition}

\begin{proof}
A tedious, but straightforward computation 
reveals that\footnote{Given a metric, the Hessian of a function is the bilinear form on the tangent bundle defined by $\hes V(q) (X_q,Y_q)=\langle \nabla_X \grad V, Y\rangle(q)$, where $X$ and $Y$ are vector field extensions of $X_q,Y_q$. Via the musical isomorphisms the Hessian also induces a bilinear form on the cotangent bundle and a pairing between the tangent and cotangent bundle, which we all denote with the same symbol.}
\begin{equation}
\label{eq:tedious}
\begin{split}
X_H(f)(q,\theta_q)=\frac{|\grad V(q)|^2}{1+|\grad V(q)|^2}&-\frac{\hes V(q)(\theta_q,\theta_q)}{1+|\grad V(q)|^2}\\  &+\frac{2\theta_q(\grad V(q))\hes V(q)(\grad V, \theta_q)}{(1+|\grad V(q)|^2)^2}.
\end{split}
\end{equation}
The reverse triangle inequality, and Cauchy-Schwarz directly give
\begin{align*}
X_H(f)
&\geq \frac{|\grad V(q)|^2}{1+|\grad V(q)|^2}-\frac{3\Vert \hes V(q)\Vert\,|\theta_q|^2}{1+|\grad V(q)|^2}.
\end{align*}
By asymptotic regularity there exists a constant $C$ such that $\frac{3\Vert \hes V(q)\Vert}{1+|\grad V(q)|^2}\leq C$, hence
\begin{equation*}
X_H(f)\geq \frac{|\grad V(q)|^2}{1+|\grad V(q)|^2}-C\,|\theta_q|^2.
\end{equation*}
This yields the following global estimate
$$
\Theta_x(X_H)(q,\theta_q) \ge \bigl( 1-\kappa C\bigr) |\theta_q|^2 + \kappa \frac{|\grad V(q)|^2}{1+|\grad V(q)|^2}>0\\
\quad \hbox{for all}~~ x\in T^*M,
$$
for all $0<\kappa\le  \kappa_0 = 1/2C$. The final step is to establish a uniform positive  lower bound on $a_\kappa$ for $(q,\theta_q)\in \Sigma_\epsilon$, independent of $(q,\theta_q)$ and $\epsilon$.

If $d_M(q,K)\geq R$ is sufficiently large, then asymptotic regularity gives that $|\grad V(q)|>\asympconstant$. Thus, in this region,
$$
 \frac{1}{2} |\theta_q|^2 + \kappa \frac{|\grad V(q)|^2}{1+|\grad V(q)|^2}\geq  \kappa \frac{|\grad V(q)|^2}{1+|\grad V(q)|^2}\geq \frac{\kappa \asympconstant^2}{1+\asympconstant^2}.
$$
On $d_M(q,K)< R$ we can use standard compactness arguments. For $(q,\theta_q)\in \Sigma_\epsilon$, we have the energy identity $ \frac{1}{2}|\theta_q|^2+V(q)=\epsilon.$ Suppose that $\frac{1}{2}|\theta_q|^2<\epsilon_0$, then $|V(q)|<\epsilon+\epsilon_0$. If $\epsilon_0$ is sufficiently small, this implies that $|\grad V(q)|\geq V_0>0$ for some constant $V_0$, because $\grad V\not=0$ at $V(q)=0$. Therefore in this case
$$
 \frac{1}{2} |\theta_q|^2 + \kappa \frac{|\grad V(q)|^2}{1+|\grad V(q)|^2}>\frac{\kappa V_0^2}{1+V_0^2}.
$$
If $|\theta_q|^2\geq \epsilon_0$, then $\epsilon_0$ is a lower bound of this quantity. We have exhausted all possibilities and established a uniform lower bound on $\Theta(X_H)$.
\end{proof}

The following a priori bounds are due to the uniform contact type of $\Sigma$.
\begin{lemma}
Let $(c_\epsilon,\tau_\epsilon)$ be critical points of $\A_\epsilon$, with 
$0< a_1 \leq \A_\epsilon(\ce,\te) \leq a_2$. Then there is a constant $T_2$, independent of $\epsilon$,
such that $\tau_\epsilon \leq T_2$ for sufficiently small $\epsilon$. 
\label{l:teupperbound}
\end{lemma}

\begin{proof}

We start with the case $\te \ge 0$. The Hamiltonian action satisfies 
$$
\A^H_\epsilon(\gamma_\epsilon,\te) = \int_{\gamma_\epsilon} \Lambda + 
\epsilon\bigl(e^{-\te} + e^{\te/2}\bigr)\le a_2,
$$
and thus $ \int_{\gamma_\epsilon} \Lambda\le a_2$.
Since $\Sigma$ is of uniform contact type it holds for $\gamma_\epsilon \subset \Sigma_{\tilde \epsilon}$, with $\tilde \epsilon\le \epsilon\le \epsilon_0$, that
$$
a_2 \ge \int_{\gamma_\epsilon} \Lambda = \int_{\gamma_\epsilon} \Theta = \int_0^{e^\te} \alpha_{\gamma_\epsilon}(X_H) \ge a_\kappa e^\te.
$$
We conclude that always $\te \le \max\{0,\log(a_2/a_\kappa)\}$, which proves the lemma.
\end{proof}

We can also establish a lower bound on $\te$ under the condition that $M$ is asymptotically flat. It is here that the assumption of $M$ having flat ends is really necessary, all other 
estimates carry through under the weaker assumption of bounded geometry. 

\begin{lemma}
Let $(\ce,\te)$ be critical points of $\A_\epsilon$, with $0< a_1 \leq \A_\epsilon(\ce,\te) \leq a_2$. If $(M,g)$ is asymptotically flat, then  there is a constant $T_3$, independent of $\epsilon$, such that $\te\geq T_3$ for sufficiently small $\epsilon$. 
\label{l:telowerbound}
\end{lemma}
\begin{proof}
Assume, by contradiction that $\te\to -\infty$ as $\epsilon \to 0$. Then Equation~\bref{e:es1} gives
\begin{equation}
\label{eq:asymptoticenergy}
2\cE(\ce) = e^\te a_\epsilon - 2\epsilon   - \frac{\epsilon}{2}
  e^{3\te/2} \to 0,\quad
  \hbox{as}\quad\epsilon \to 0.
\end{equation}
Fix $s_0\in \mS^1$. Then the previous equation implies, using Equation\ \ref{eqn;metric-1}, that $\ce(s)\in B_{\epsilon'}\bigl(\ce(s_0)\bigr)$, where $\epsilon' = \sqrt{e^\te a_2 - 2\epsilon   - \frac{\epsilon}{2} e^{3\te/2}}$. We distinguish two cases: 

(i) There exists an $R>0$ such that $d_M(\ce(s_0),K)\le R$ for all $\epsilon$. Then  
$\ce(s) \in B_{\epsilon'+R}(K)$, and therefore $|V(\ce(s))|\le C$ for all $s\in \mS^1$ and all $\epsilon>0$. This implies $\int_0^1 e^\te V(\ce(s))ds \to 0$, which contradicts (\ref{e:es1b}), as $a_\epsilon>0$, and thus $\te \ge T_3$. 

(ii) Now we assume no such $R>0$ exists, and assume thus that $d_M(\ce(s_0),K) \to \infty$ as $\epsilon \to 0$ to derive a contradiction. By bounded geometry of $M$, every point $q\in M$ has a normal charts $(\cU_q,\exp_q^{-1})$ and constants $\rho_0,R_0>0$ such that $B_{\rho_0}(q)\subset \cU_q$ and $|\partial^\ell\Gamma^k_{ij}(q)|\le R_0$. This implies that $\ce(s)\in \cU_{\ce(s_0)}$ for sufficiently small $\epsilon$. We assume $M$ has flat ends, and since $d(\ce(s_0),K)\rightarrow \infty$ the metric on the charts $\cU_{\ce(s_0)}$ is flat. We identify these charts with open subsets of $\mR^n$ henceforth. The differential equation $\ce$ satisfies is
\begin{equation}
\label{eq:ce}
e^{-2\te} \nabla_s c'_\epsilon(s)+\grad V(\ce(s))=0.
\end{equation}
Take the unique geodesic $\gamma$ from $\ce(s_0)$ to $\ce(s)$ parameterized by arc length, i.e.
$$\gamma(0)=\ce(s_0),\quad  \gamma(d_M(\ce(s_0),\ce(s))=\ce(s),\quad\text{and}\quad |\gamma'(t)|=1.$$ 
Then, by asymptotic regularity, $\Vert \hes V(\gamma(t))\Vert\leq C|\grad V(\gamma(t))|$ for some constant $C>0$, and
\begin{equation*}
\begin{split}
\frac{d}{dt}|\grad V(\gamma(t))|^2&=2\, \hes V(\gamma(t))(\grad V(\gamma(t)),\gamma'(t))\\
&\leq 2 \Vert \hes V(\gamma(t))\Vert |\grad V(\gamma(t))| \leq 2 C|\grad V(\gamma(t))|^2.
\end{split}
\end{equation*}
Gronwall's inequality therefore implies that 
\begin{equation}
\label{eq:gronwall}
|\grad V(\gamma(t))|\leq |\grad V(\gamma(0))|e^{C t}.
\end{equation} 
We identify $U_{\ce(s_0)}$ with an open subset of $\mR^n$, and we write $
\grad V(\gamma(t))=\grad V(\gamma(0))+\int_0^t \frac{d}{d\sigma}\grad V(\gamma(\sigma))d\sigma.
$
Hence
\begin{align}
\nonumber|\grad V(\gamma(t))-\grad V(\gamma(0))|&\leq \int_0^t\Vert\hes V(\gamma(\sigma))\Vert d\sigma\leq C\int_0^t|\grad V(\gamma(\sigma))|d\sigma\\
\label{eq:ceestimate}
&\leq |\grad V(\gamma(0))|(e^{Ct}-1).
\end{align}
For any solution to Equation~\bref{eq:ce}, we compute
\begin{equation*}
\begin{split}
\frac{d}{ds}e^{2\te}\langle \grad V(\ce(s_0)),c'_\epsilon(s)\rangle=&e^{2\te}\langle \grad V(\ce(s_0)),\nabla_s c'_\epsilon(s)\rangle\\
=&-\langle \grad V(\ce(s_0)),\grad V(\ce(s))\rangle \\
=&-\langle \grad V(\ce(s_0)),\grad V(\ce(s_0))\rangle\\
&\,-\langle \grad V(\ce(s_0)),\grad V(\ce(s))-\grad V(\ce(s_0))\rangle
\end{split}
\end{equation*}
By asymptotic regularity and Estimate~\bref{eq:ceestimate}, we find that
$$
\frac{d}{ds}e^{2\te}\langle \grad V(\ce(s_0),c'_\epsilon(s)\rangle\leq-\asympconstant^2+\asympconstant^2(e^{C d_M(\ce(s_0),\ce(s))}-1).
$$
We see that as $\epsilon\rightarrow 0$ that $e^{2\te}\langle \grad V(\ce(s_0),c'_\epsilon(s)\rangle$ is monotonically decreasing in $s$. We conclude that $\ce$ cannot be periodic. This is a contradiction, therefore there exists a constant $T_3$ such that $\te\geq T_3$, for sufficiently small $\epsilon$.
\end{proof}

\begin{proposition}
\label{prop:pen}
Let $(\ce,\te)$, $\epsilon \to 0$ be a sequence satisfying $d\A_\epsilon(\ce,\te) =0$, and $0<a_1\le \A_\epsilon(\ce,\te) \le a_2$. If $(M,g)$ has flat ends then there exists a convergent subsequence $(c_{\epsilon'},\tau_{\epsilon'}) \to (c,\tau)$ in $\Lambda M\times\mR$, $\epsilon'\to 0$. The limit satisfies $d\A(c,\tau)=0$, and $0<a_1\leq\A(c,\tau)\leq a_2$.
\end{proposition}

\begin{proof}
From Lemmas \ref{l:teupperbound} and \ref{l:telowerbound} we obtain uniform bounds on $\te$. We
can now repeat the arguments of the proof of Proposition \ref{lem:PS} on the sequence $\{\ce\}$, from which we draw the desired conclusion.
\end{proof}

\section{The relation of the topology of the hypersurface with the topology of its projection}
\label{sec:hyp}

We investigate the relation between the topology of $\hs$ and its projection $N=\pi(\Sigma)$ to the base manifold. Recall that we assume $H$ to be mechanical and that the hypersurface $\hs=H^{-1}(0)$ is regular. Thus $N$ and its boundary $\partial N$ are given by
$$
N=\{q\in M~|~V(q)\leq 0\},\quad\text{and}\quad\partial N=\{q\in M~|~V(q)= 0\},
$$
and $\partial N$ is smooth. We have the topological characterization
\begin{equation}
\hs\cong ST^*N\bigcup_{ST^*N\bigr|_{\partial N}} DT^*N\bigr|_{\partial N}.
\end{equation}
The characterization is given in terms of the sphere bundle $ST^*N$ and the disc bundle $DT^*N$ in the cotangent bundle of $N$. The vertical bars denote the restriction of the bundles to the boundary. This topological characterization gives a relation between the homology of $\hs$ and $N$. In this section we identify $\Sigma$ with this characterization.

Recall that a map is proper if preimages of compact sets are compact. In the proof of the next proposition, compactly supported cohomology $H_c^*(M)$ is used, which is contravariant with respect to proper maps. In singular (co)homology, homotopic maps induce the same maps in (co)homology. For compactly supported cohomology, maps that are homotopic via a homotopy of proper maps, induce the same maps in cohomology. If $\partial N=\emptyset$ the following proposition directly follows from the Gysin sequence.
\begin{proposition}
\label{prop:homologyrelation}
There exist isomorphisms $H^i_c (\hs)\cong H_c^i(N)$ for all $0\leq i \leq n-2$. 
\end{proposition}
\begin{proof}
Let $C$ be the closure of a collar of $\partial N$ in $N$. Thus $C$ deformation retracts via a proper homotopy onto $\partial N$. Denote by $\pi$ also the projection $ST*M\rightarrow M$. Then $\pi^{-1}(C)$ is the closure of a collar of $ST^*N\bigr|_{\partial N}=\partial ST^*N$ in $ST^*N$, and therefore it deformation retracts via a proper homotopy onto $ST^*N\bigr|_{\partial N}$. Define $D\subset \Sigma$ by
\begin{equation*}
D=\pi^{-1}(C)\bigcup_{ST^*N\bigr|_{\partial N}}DT^*N\bigr|_{\partial N}.
\end{equation*}
This is a slight enlargement of the disc bundle of $M$ restricted to the boundary $\partial N$, which Figure~\ref{fig:bundles} clarifies. By construction $D$ deformation retracts properly to $DT^*N\bigr|_{\partial N}$, which in turn deformation retracts properly to $\partial N$. This induces an isomorphism
\begin{equation}
\label{eq:easyiso}
H_c^*(D)\cong H_c^*(\partial N).
\end{equation}
Let $S=ST^*N$. The intersection $D\cap S$ deformation retracts properly to $ST^*N\bigr|_{\partial N}$. Thus the isomorphism
\begin{equation}
\label{eq:easyiso2}
H^*_c(D\cap S)\cong H^*_c(ST^*N\bigr|_{\partial N}),
\end{equation}
holds. The inclusions in the diagram
\begin{equation*}
\xymatrix@R=.2cm{
&S\ar[rd]^{\jmath_1}&\\
S\cap D\ar[ru]^{\imath_1}\ar[rd]_{\imath_2}&&\Sigma\\
&D\ar[ru]_{\jmath_2}&
}
\end{equation*}
are proper maps, because the domains are all closed subspaces of the codomains. This gives rise to the contravariant Mayer-Vietoris sequence of compactly supported cohomology of the triad $(\hs, S,D)$ 
\begin{equation}
\label{eq:mayervietorisN}
\xymatrix@=2.1pc{\ar[r]&
H_c^i(\hs)\ar[r]^-{(\jmath_1^i,-\jmath_2^i)}&H_c^i(S)\oplus H_c^i(D)\ar[r]^-{\imath_1^i+\imath_2^i}&H_c^i(S\cap D)\ar[r]&H^{i+1}_c(\hs)\ar[r]&\\
}
\end{equation}
The map $\imath_2^i$ is an isomorphism: this can be seen from the Gysin sequence for compactly supported cohomology as follows. Recall that from any vector bundle $E\rightarrow B$ of rank $n$ over a locally compact space $B$, we can construct a sphere bundle $SE\rightarrow B$. The Gysin sequence relates the cohomology of $SE$ and $B$,
\begin{equation}
\label{eq:ggysin}
\xymatrix{
\ldots\ar[r]&H^{i-n}_c(B)\ar[r]^-{\epsilon^i}&H^i_c(B)\ar[r]^{\pi^i}&H^i_c(SE)\ar[r]^-{\delta}&H^{i-n+1}_c(B)\ar[r]&\ldots
}
\end{equation}
The map $\epsilon^i$ is the cup product with the Euler class of the sphere bundle. We apply this sequence to the sphere bundle in $T^*N$ restricted to $\partial N$. For dimensional reasons the sequence breaks down into short exact sequences
\begin{equation}
\label{eq:gysindN}
\xymatrix{
0\ar[r] &H_c^i(\partial N)\ar[r]^-{\pi^i}&H_c^i(ST^*N\bigr|_{\partial N})\ar[r]&0\qquad \text{for }\qquad 0\leq i \leq n-2.
}
\end{equation}
The diagram
\begin{equation*}
\xymatrix{
H^i_c(D)\ar[r]^-{\imath_2^i}\ar[d]_{\cong}&H_c^i(S\cap D)\ar[d]^{\cong}\\
H^i_c(\partial N)\ar[r]^-{\pi^i}&H^i_c\left(ST^*N\bigr|_{\partial N}\right)
}
\end{equation*}
commutes. This shows that $\imath_2^i$ is an isomorphism for $0\leq i\leq n-2$. The map $\imath_1^i+\imath_2^i$ in the Mayer-Vietoris sequence, Equation~\bref{eq:mayervietorisN}, is surjective, and the sequence breaks down into short exact sequences
\begin{equation}
\label{eq:shortN}
\xymatrix{0\ar[r]&H_c^i(\hs)\ar[r]&H_c^i(S)\oplus H_c^i(D)\ar@<3pt>[r]^-{\imath_1^i+\imath_2^i}&H_c^i(S\cap D)\ar@<3pt>[l]^-{p}\ar[r]&0}.
\end{equation}
More is true, since the sequence actually splits by the map $p=(0,(\imath_2^i)^{-1})$. If we study the Gysin sequence for $N$ and $S$ we see that
\begin{equation}
\label{eq:gysinN}
\xymatrix{
0\ar[r]&H^i_c(N)\ar[r]^-{\pi^i}&H_c^i(S)\ar[r]&0,
}\quad\text{for}\quad 0\leq i\leq n-2.
\end{equation}
The isomorphisms~\bref{eq:gysinN},~\bref{eq:easyiso}, \bref{eq:easyiso2}, and~\bref{eq:gysindN} can be applied to the sequence in Equation~\bref{eq:shortN} 
and because the sequence is split the stated isomorphism holds.
\end{proof}
\begin{figure}
\def\svgwidth{.6\textwidth}
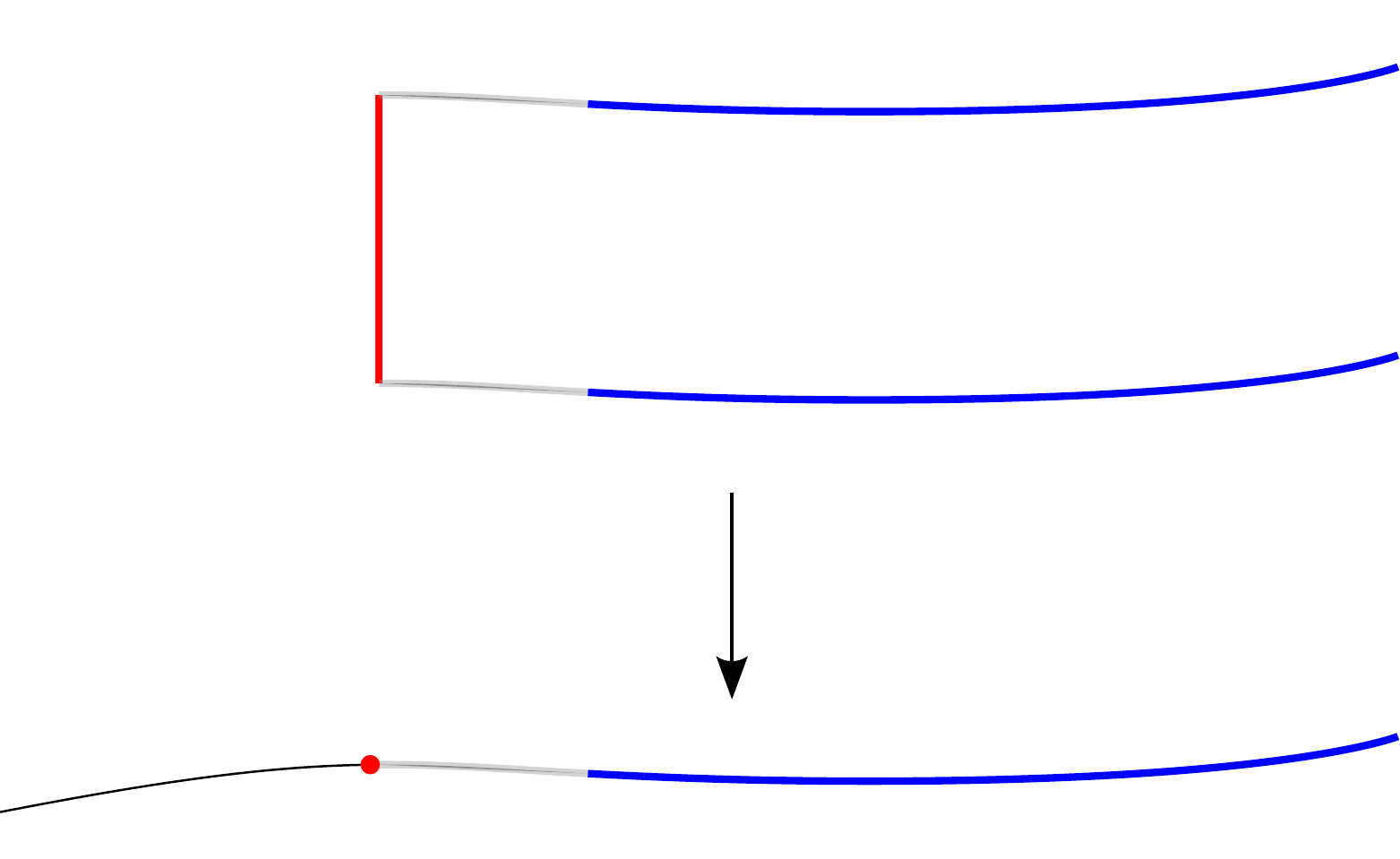
\caption{A sketch of the spaces $D$, and $S$. In the picture $N$ is a half-line, hence $\partial N$ is a point. The topology of the energy hypersurface can be recovered from its projection $N$.
}
\label{fig:bundles}
\end{figure}
\begin{proposition}
\label{prop:duality}
For all $2\leq i\leq n$ there is an isomorphism
\begin{equation}
H_{i}(N,\partial N)\cong H_{i+n-1}(\hs).
\end{equation}
\end{proposition}
\begin{proof}
This is a double application of Poincar\'e duality for non-compact manifolds with boundary. 
The dimension of $N$ is $n$, and therefore Poincar\'e duality gives $H_{i}(N,\partial N)\cong H_c^{n-i}(N)$. The boundary of $\hs$ is empty, and its dimension equals $2 n-1$, thus $H_{n+i-1}(\hs)\cong H^{n-i}_c(\hs)$. By Proposition~\ref{prop:homologyrelation} we have $H^{n-i}_c(N)\cong H^{n-i}_c(\hs)$, for all $2\leq i\leq n$. The isomorphism stated in the proposition is the composition of the isomorphisms.
\end{proof}

We would also like the previous proposition to be true if $i=1$. This is the case if the bundle $ST^*N$ is trivial, but in general this is not true. However, the following result is sufficient for our needs.

\begin{proposition}
\label{prop:exceptional}
If $H_{n}(\Sigma)\not=0$ and $H_{n}(M)=0$, then $H_1(N,\partial N)\not=0$.
\end{proposition}
\begin{proof}
We will show that a non-zero element in $H_c^{n-1}(\Sigma)$ gives rise to a non-zero element in $H_c^{n-1}(N)$. A double application of Poincar\'e duality, as in the previous proposition, will give the desired result. We will use the same notation as in the proof of Proposition~\ref{prop:homologyrelation}. The Gysin sequence, Equation\ \bref{eq:ggysin}, for the sphere bundle $ST^*N\bigr|_{\partial N}$ over $\partial N$ breaks down to the short exact sequence
\begin{equation}
\label{eq:gysinnn}
\xymatrix{0\ar[r]&H_c^{n-1}(\partial N)\ar[r]^-{\pi^{n-1}}&H_c^{n-1}(ST^*N\bigr|_{\partial N})\ar[r]^-{\delta}&H_c^0(\partial N)\ar[r]&0}.
\end{equation}
Because $ST^*N\bigr|_{\partial N}$ is an $(n-1)$-dimensional sphere bundle over an $(n-1)$-dimensional manifold, it admits a section $\sigma:\partial N\rightarrow ST^*N\bigr|_{\partial N}$ and Equation\ \bref{eq:gysinnn} splits. We obtain the isomorphism
$$
H_c^{n-1}(S\cap D)\cong H_c^{n-1}(ST^*N\bigr|_{\partial N})\cong H_c^{n-1}(\partial N)\oplus H_c^0(\partial N).
$$
where the first isomorphism is induced by a homotopy equivalence. Now we look at the Mayer-Vietoris sequence for $S,D$, 
$$
\xymatrix@C=1.3cm{0\ar[r]&H_c^{n-1}(\Sigma)\ar[r]^-{(\jmath_1^{n-1},-\jmath_2^{n-1})}&H_c^{n-1}(S)\oplus H_c^{n-1}(D)\ar[r]^-{\imath_1^{n-1}+\imath_2^{n-1}}\ar[d]^{\cong}&H_c^{n-1}(S\cap D)\ar[d]^{(\sigma^{n-1},\delta)}\\
&&H_c^{n-1}(S)\oplus H_c^{n-1}(\partial N)\ar[r]&H_c^{n-1}(\partial N)\oplus H_c^0(\partial N).}
$$
We get a zero on the left of this sequence, because we have shown that in the previous step that the map $\imath_1^{n-2}+\imath_2^{n-2}$ is surjective, cf. the argument before Equation~\bref{eq:shortN}. We claim that $\jmath_1^{n-1}$ is injective. Suppose otherwise, then there are $\bx,\by\in H_c^{n-1}(\Sigma)$, with $\bx\not=\by$ such that $\jmath_1^{n-1}(\bx)=\jmath_1^{n-1}(\by)$. Then $j^{n-1}_1(\bx-\by)=0$. But since the map $(\jmath_1^{n-1},-\jmath_2^{n-1})$ is injective, we realize that $\jmath_2^{n-1}(\bx-\by)\not=0$. But then $\imath_2^{n-1}\jmath_2^{n-1}(\bx-\by)=0$ by the exactness of the sequence. Moreover
$$\sigma^{n-1}\imath_2^{n-1}\jmath_2^{n-1}(\bx-\by)=(\imath_2\sigma)^{n-1}\jmath_2^{n-1}(\bx-\by).$$ 
But, by the proper homotopy equivalence $D\cong \partial N$, we realize that $(\imath_2\sigma)^{n-1}:H_c^{n-1}(D)\rightarrow H_c^{n-1}(\partial N)$ is an isomorphism, and $\jmath_2^{n-1}(\bx-\by)\not=0$. This is a contradiction, hence $\jmath_1^{n-1}$ is injective.
Recall that the Gysin sequence comes from the long exact sequence of the disc and sphere bundle, and the Thom isomorphism. From this we derive the following commutative diagram, which shows a naturality property of the Gysin sequence.
$$
\xymatrix@C=.4cm{
0\ar[r]&H_c^{n-1}(\partial N)\ar[r]&H_c^{n-1}(ST^*N\bigr|_{\partial N})\ar[r]^-\delta&H_c^0(\partial N)\ar[r]&\\
0\ar[r]&H_c^{n-1}(DT^*N\bigr|_{\partial N})\ar[r]\ar[u]_{\cong}&H_c^{n-1}(ST^*N\bigr|_{\partial N})\ar[r]^-\delta\ar[u]_{=}&H_c^n(DT^*N\bigr|_{\partial N},ST^*N\bigr|_{\partial N})\ar[r]\ar[u]_{\Phi^{-1}\,\cong}&\\
0\ar[r]&H_c^{n-1}(DT^*N)\ar[r]\ar[u]&H_c^{n-1}(ST^*N)\ar[r]^-\delta\ar[u]_{\imath_1^{n-1}}&H_c^n(DT^*N,ST^*N\bigr)\ar[r]\ar[u]&\\
0\ar[r]&H_c^{n-1}(N)\ar[r]\ar[u]_\cong&H_c^{n-1}(ST^*N)\ar[r]^-\delta\ar[u]_=&H_c^0(N)\ar[r]\ar[u]_{\Phi\,\cong}&\\}
$$
The top and bottom rows are the Gysin sequences of $(\partial N,ST^*N\bigr|_{\partial N})$ and $(N,ST^*N\bigr|_{N})$ respectively. The vertical maps between the middle rows are the pullback maps of the inclusion of pairs $(DT^*N\bigr|_{\partial N},ST^*N\bigr|_{\partial N})\rightarrow (DT^*N,ST^*N\bigr|_{N})$. The vertical maps $\Phi$ are the Thom isomorphisms. The map $\imath_1^{n-1}$ in the diagram is is the same as the map induced by $\imath_1:S\cap D\rightarrow S$, under the isomorphism induced by the homotopy equivalence $S\cap D\cong ST^*N\bigr|_{\partial N}$, which we therefore denote by the same symbol. 

We want to show that $\delta \jmath_1^{n-1}(y)=0$ for all $y\in H_c^{n-1}(\Sigma)$. For this we argue as follows. Recall that $H_c^0(N)$ consists of constant functions of compact support, and therefore is generated by the number of compact components of $N$. If the vertical map in the third column, from $H_c^0(N)\rightarrow H_c^0(\partial N)$ is not injective, then $N$ has a compact component without boundary. This implies that $M$ must have a compact component without boundary. But we assume that $H_n(M)=0$, therefore $M$ does not have orientable compact components and $H_c^0(N)\rightarrow H_c^0(\partial N)$ is injective. Let $\by\in H_c^0(\Sigma)$ be non-zero. Obviously in $H_c^{n-1}(ST^*N\bigr|_{\partial N})$ we have the equality $\imath_1^{n-1}\jmath_1^{n-1}(\by)=\imath_2^{n-1}\jmath_2^{n-1}(\by)$, and from the definition of the boundary map in the long exact sequence of the pair in the second row of the diagram, we obtain
\begin{align*}
\delta \imath_1^{n-1}\jmath_1^{n-1}(\by)&=\delta \imath_2^{n-1}\jmath_2^{n-1}(\by)=[p^{-1}d(\imath_2^{n-1})^{-1}\imath_2^{n-1}\jmath_2^{n-1}{\boldsymbol y}]\\&=[p^{-1}d\jmath_2^{n-1}{\boldsymbol y}]=[p^{-1}\jmath_2^{n-1}d{\boldsymbol y}]=0.
\end{align*}
where $p$ is the projection map in the defining short exact sequence. By the injectivity of the map $H_c^0(N)\rightarrow H_c^0(\partial N)$, and the commutativity of the diagram we must have that $\delta \jmath_1^{n-1}(\by)=0\in H_c^0(N)$. The exactness of the bottom row now shows that there must be an element in $H_c^{n-1}(N)$ which is mapped to $\jmath_1^{n-1}(\by)$, because it $\jmath_1^{n-1}(\by)$ is in the kernel of $\delta$. Poincar\'e duality for non-compact manifolds with boundary states that $H_{n}(\Sigma)\cong H_c^{n-1}(\Sigma)$, and $H_c^{n-1}(N)\cong H_{1}(N,\partial N)$. Thus, by the preceding argument we get a non-zero class in $H_1(N,\partial N)$.
\end{proof}

\begin{proposition}
Suppose that $H_{k+n}(\hs)\not=0$ and $H_{k+1}(M)=0$, for some $0\leq k\leq n-1$. Then there exists a non-zero class in $H_{k}(M\setminus N)$ which is mapped to zero in $H_{k}(M)$ by the morphism induced by the inclusion.
\label{prop:nonzerohomologya}
\end{proposition}
\begin{proof}
 Consider the long exact sequence of the pair $(M,M\setminus N)$
\begin{equation*}
H_{k+1}(M)\rightarrow H_{k+1}(M,M\setminus N)\rightarrow H_k(M\setminus N)\rightarrow H_k(M).
\end{equation*}
The homology group $H_{k+1}(M)$ is zero by assumption, thus the middle map is injective. If we can find a non-zero element in $H_{k+1}(M,M\setminus N)$, then we see it is mapped to a non-zero element of $H_k(M\setminus N)$, which in turn is mapped to zero in $H_k(M)$ by exactness of the sequence. By excision $H_{k+1}(M,M\setminus N)\cong H_{k+1}(N,\partial N)$. Thus there is a non-zero element of $H_k(M\setminus N)$ which is mapped to zero in $H_k(M)$ by the inclusion for $0\leq k\leq n-1$. \end{proof}

\begin{remark}
In our setting, the assumption $H_{k+1}(M)=0$ is automatically satisfied. This follows from the assumption $H_{k+1}(\Lambda M)=0$ on the topology of the loop space, and Equation~\bref{eq:splithomology}.
\end{remark}

\section{The link}
\label{sec:link}
\subsection{The parameter $\nu$}
For analytical reasons, we need to shrink the set $N=\pi(\Sigma)=\{q\in M\,|\,V(q)\leq 0\}$ to 
\begin{equation}
\nN=\{q\in M\,|\, V(q)\leq -\nu\sqrt{1+|\grad V(q)|^2}\}.
\end{equation}
For small $\nu$ this can be done diffeomorphically. On the modified set $N_\nu$ we estimate the potential $V$ uniformly. 

\begin{lemma}
\label{lem:modifyN}
There exist $\nu>0$ sufficiently small, such that
\begin{itemize}
\item The spaces $N$ and $\nN$ are diffeomorphic, and $M\setminus N$ and $M\setminus \nN$ are diffeomorphic. 
\item If $H_{k+n}(\Sigma)\not=0$ and $H_{k+1}(M)=0$ for some $k$, there exists a non-zero class in $H_k(M\setminus \nN)$ which is mapped to zero in $H_k(M)$ by the morphism induced by the inclusion. 
\item There exists a $\rho_\nu>0$ such that, for all $q\in N_\nu$,
\begin{equation}
V(\tilde q)\leq-\frac{\nu}{2},\qquad \text{for all}\qquad \tilde q\in B_{\rho_\nu}(q).
\end{equation}
\end{itemize}
\end{lemma}
\begin{proof}
Consider the function $f:M\rightarrow \mR$ defined by
\begin{equation*}
\label{eq:f}
f(q)=\frac{V(q)}{\sqrt{1+|\grad V(q)|^2}}.
\end{equation*}
The gradient flow of this function induces the diffeomorphism. Because $N$ is non-compact, the Gradient Deformation Lemma does not apply. However, following the estimates of Lemma 14 of~\cite{bpv}, it follows that this function satisfies the condition of Palais and Smale, and has no critical values between $0$ and $-\nu$. A theorem of Palais~\cite[Theorem 10.2]{palaishilbert} now shows that $N$ and $\nN$ are diffeomorphic through the gradient flow defined by this function. 
Proposition~\ref{prop:nonzerohomologya} therefore shows that there exists a non-zero class in $H_{k+1}(M\setminus \nN)$ that is mapped to zero in $H_k(M)$ by the morphism induced by the inclusion.
We now estimate $V$ uniformly on balls of radius $\rho_\nu$ around points of $N_\nu$. By continuity and compactness, there exists a $\rho_\nu>0$ such that for all $q\in N_\nu$ with $d(q,K)<1$, and all $\tilde q\in B_{\rho_\nu}(q)$, the estimate $V(\tilde q)\leq -\frac{\nu}{2}$ holds. 
For $q\in N_\nu$ with $d(q,K)\geq 1$, if $\rho_\nu<1$, then $d(\tilde q,K)>0$ for all $\tilde q\in B_{\rho_\nu} (q)$.
If $\rho_\nu<\inj M$, there exists a unique shortest geodesic $\gamma$ parameterized by arclength $\rho_\nu'<\rho_\nu$, and Equation~\bref{eq:gronwall} holds. We compute
\begin{equation*}
\begin{split}
V(\tilde q)=V(q)+\int_0^{\rho_\nu'}\frac{d}{ds} V(\gamma(s))ds,
&\leq V(q)+|\grad V(q)|\frac{e^{C\rho_\nu'}-1}{C} \\&\leq-\nu \sqrt{1+|\grad V(q)|^2}+|\grad V(q)|\frac{e^{C\rho_\nu}-1}{C}.
\end{split}
\end{equation*}
The function $x\mapsto-\nu \sqrt{1+x^2}+\frac{e^{C\rho_\nu}-1}{C} x$ has the maximum $-\sqrt{\nu^2-\left(\frac{e^{2 C\rho_\nu}-1}{C}\right)^2}$ for $\frac{e^{2 C\rho_\nu}-1}{C}\leq \nu$. We can find $\rho_\nu>0$ small such that $V(\tilde q)\leq -\frac{\nu}{2}$. This is independent of $q$, because so is $C$.
\end{proof}
\begin{figure}
\def\svgwidth{.4\textwidth}
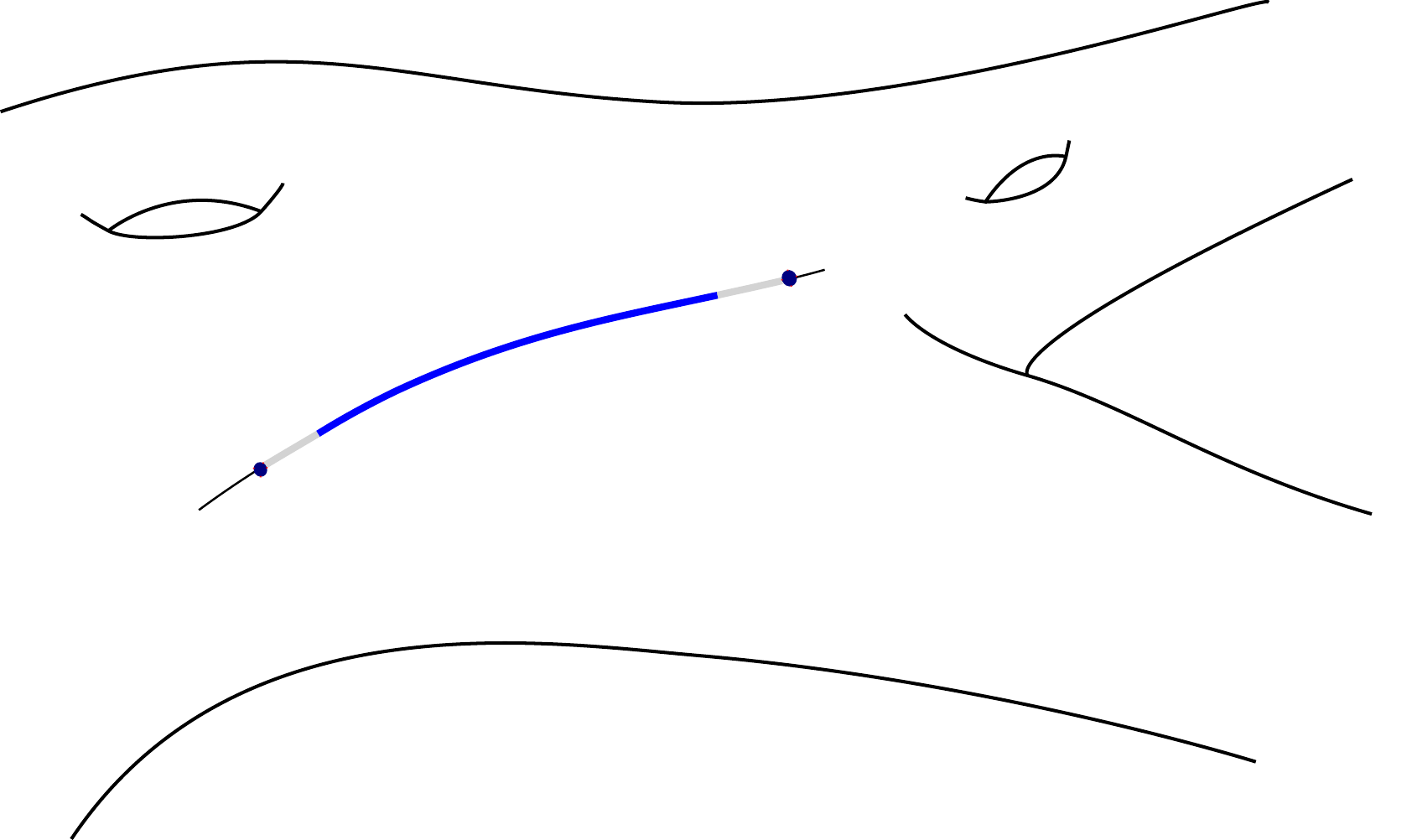
\def\svgwidth{.4\textwidth}
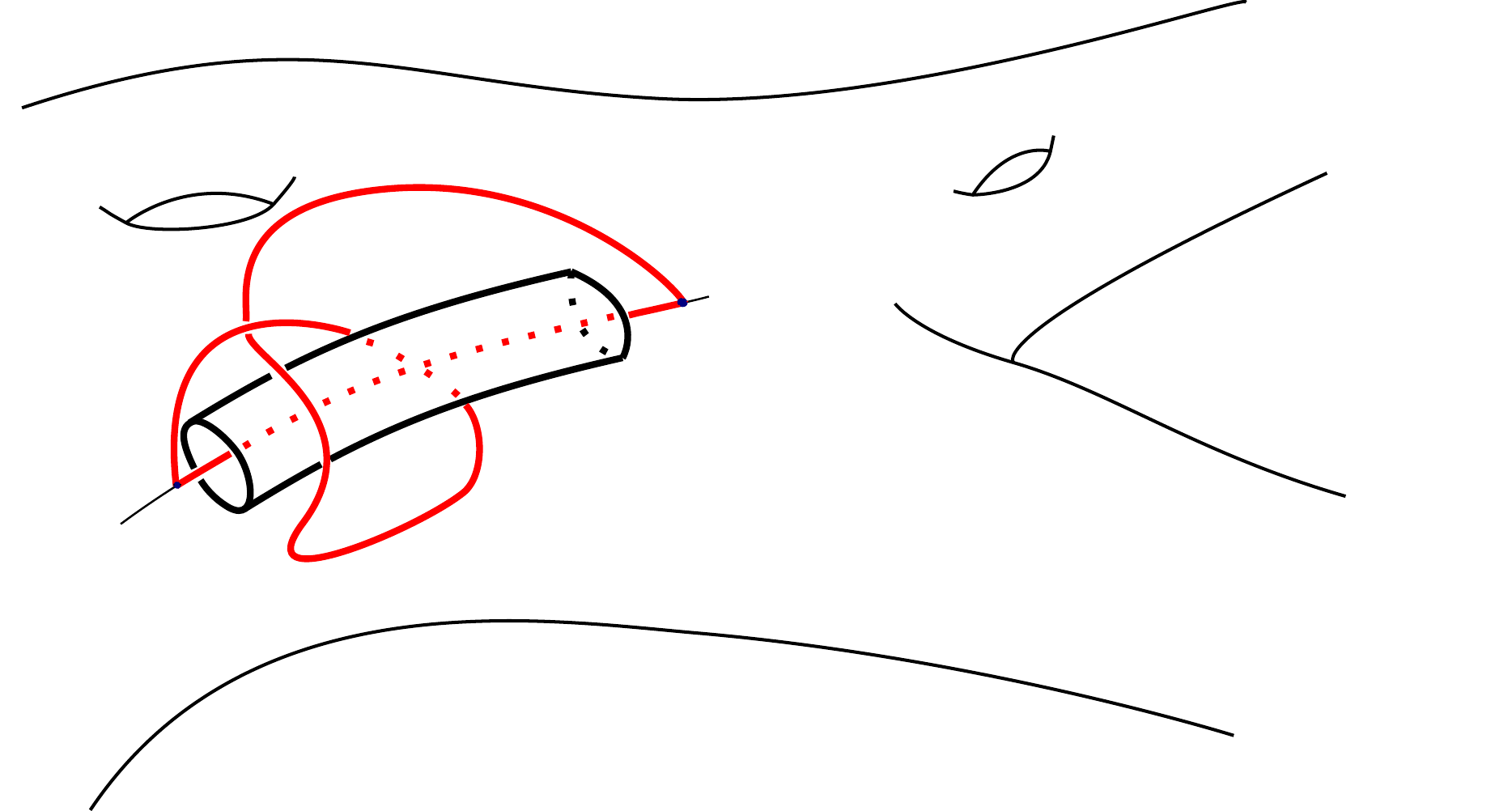
\caption{Sketches of the domain $\Lambda M\times\mR$ of the functional $\A$. The manifold $M$ is embedded in $\Lambda M$ by the map sending $q\in M$ to the constant loop $c_q(s)=q$, and hence is embedded in $\Lambda M\times\mR$. On the left, the $(k)$-link in the base manifold, between $W$ and $N_\nu$ is shown. This link obviously does not persist in the loop space. However, it is possible to \emph{lift} the link to $\Lambda M\times\mR$, depicted on the right, to the sets $A$ and $B$, which $(k+1)$-link in $\Lambda M\times \mR$.}
\end{figure}

\subsection{Constructing linking sets.}
We will use the topological assumptions in  Theorem \ref{thm:main2}, to construct linking subspaces of the loop space. These are in turn used to find candidate critical values of the functional $\A$. 

The goal is to construct a link in the function space with right bounds for $\A$ on the linking sets. We will in fact construct this link in a tubular neighborhood of $M$ in $\Lambda M$.

By the assumption of bounded geometry, we can construct a well behaved tubular neighborhood. Let $NM$ be the normal bundle of $\inclloop(M)$ in $\Lambda M$. Recall that we denote the constant loop at $q\in M$, by $c_q$, thus $c_q(s)=q$ for all $s\in \mS$. Elements $\xi\in N_{c_q}M$ are characterized by the fact that $\int_{\mS^1} \xi(s)ds=0$. Assuming that $M$ is of bounded geometry, we get a uniform tubular neighborhood in the loop space. 

\begin{proposition}
\label{propo:tubneighborhood}
Assume that $M$ is of bounded geometry. Then there exists an open neighborhood $\tub$ of $\inclloop(M)$ in $\Lambda M$ and a diffeomorphism $\phi :NM\rightarrow \tub$, with the property that it maps $\xi\in NM$ with $\Vert \xi\Vert_{H^1}\leq \frac{\inj M}{2}$ to $\phi(\xi)\in \Lambda M$ with $d_{H^1}(c_q,\phi(\xi))=\Vert  \xi\Vert_{H^1}$.
\end{proposition}

The inclusion of the zero section in the normal bundle is denoted by $\incl: M\rightarrow NM$. The zero section of the normal bundle is mapped diffeomorphically into $\inclloop(M)\subset \Lambda M$ by $\phi$. The norm $\Vert \cdot \Vert_{\perp}$ defined by
$\Vert \xi \Vert_{\perp}=\int_{\mS^1}\langle \triangledown \xi (s), \triangledown \xi(s)\rangle\, ds$,
is equivalent to the norm $\Vert \cdot \Vert_{H^1}$. To be precise the following estimate holds
\begin{equation}
\label{eq:normalestimate}
\Vert \xi \Vert_\perp\leq\Vert \xi \Vert_{H^1}\leq \sqrt{2}\Vert \xi \Vert_\perp.
\end{equation}

By Proposition~\ref{prop:nonzerohomologya} and Lemma~\ref{lem:modifyN} there exists a non-zero $\hw\in H_k(M \setminus N_\nu)$ such that $i_k(\hw)=0$ in $H_k(M)$. In this formula $i$ is the inclusion $i:M\setminus N_\nu\rightarrow M$, and $i_k$ the induced map in homology of degree $k$. Because $i_k\hw=0$, there exists a $\cu\in C_{k+1}(M)$ such that $\partial \cu=\cw$. We disregard any connected component of $\cu$ that does not intersect $\cw$. Set $W=|\cw|$ and $U=|\cu|$ where $|\cdot|$ denotes the support of a cycle. Both are compact subspaces of $M$. Note that we can assume that $W$ is contained in $M\setminus N$, because $M\setminus N_\nu$ is a deformation retract of $M\setminus N$. The inclusion $H_k(W)\rightarrow H_k(M\setminus N_\nu)$ is non-trivial by construction. We say that $W$ $(k)$-links $N_\nu$ in $M$.
The linking sets discussed above will be used to construct linking sets in the loop space, satisfying appropriate bounds, cf.~Proposition~\ref{prop:estimates}. A major part of this construction is carried out by the ``hedgehog'' function, which is a continuous map $\hedg:[0,1]\times U\rightarrow \Lambda M$ with the following properties
\begin{enumerate}
\item[(i)] $\hedg_0(U)\subset \tub$, with the tubular neighborhood $\tub$ defined in Proposition\ \ref{propo:tubneighborhood}.
\item[(ii)] The restriction $\hedg_t\bigr|_W$ is the inclusion of $W$ in the constant loops in $\Lambda M$,
\item[(iii)] Only $W$ is mapped to constant loops. Thus $h_t(q)\in \inclloop(M)$ if and only if $q\in W$.
\item[(iv)] $\int_0^1V(\hedg_1(q)(s))ds>0$ for all $q\in U$.
\end{enumerate}

The construction is similar, but not equivalent to the construction of such a function in the appendix of~\cite{bpv}. The reason that this construction cannot be followed ad verbatim, is that the topology of the loop space might be non-trivial and that the global interpolation operators used there cannot be defined. The construction has to be done locally: for $t=0$, a point $q\in U$ is mapped to a loop close (in $H^1$ sense) to the constant loop $c_q(s)=q$. Points on the boundary $W$ are mapped to constant loops, but other points are never mapped to a constant loop. This ensures the first three properties (for $t=0$). By construction, there are a finite number of points where the loops stay for most of the time. These points are then homotoped to points where the potential is positive. This ensures the last property, using compactness of $U$. The details of the construction are given in~\cite{Rot:ww}.

Properties (i) and (ii) are used to lift the link of $M$ to a link in $\Lambda M$. The remaining properties are used to deform the link to sets where the functional satisfies appropriate bounds, and show that the link is not destroyed during the homotopy.  

Because $U$ is compact, and $N_\nu$ is closed, $\inclloop(U\cap N_\nu)$ is compact. Moreover, it does not intersect $h_t(U)$ for any $t$, by property (iii). Hence $d_{\Lambda M}(h_{[0,1]}(U),\inclloop(U\cap N_\nu))>0$. Set $0<\rho<\min(\frac{\inj M}{2}, \frac{\rho_\nu}{2})$ such that 
\begin{equation}
  \label{eq:rho}
\rho<\frac{1}{2}d_{\Lambda M}(h_{[0,1]}(U),\inclloop(U\cap N_\nu)).
\end{equation}

\begin{figure}[t]
\def\svgwidth{.7\textwidth}
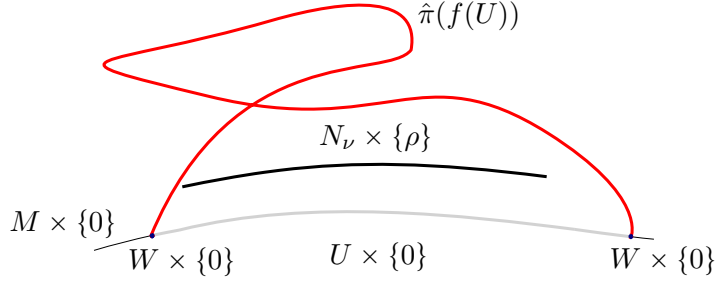

\caption{The projected normal bundle $\proj(NM)=M\times\mR$ is depicted. The set $Z=U\times\{0\}\cup \hat \pi(f(U))$ and $N_\nu\times \{\rho\}$ $(k+1)$ link in $M\times \mR$. }
\label{fig:projectednormal1}
\end{figure}

Define $f:U\rightarrow NM$ by the equation $f(q)=\phi^{-1} \hedg_0(q)$, where $\phi:NM\rightarrow \tub$ is defined in Proposition~\ref{propo:tubneighborhood}. The restriction of $f$ to $W$ is the inclusion of $W$ into the zero section of $NM$ by Property (ii). Recall that the normal bundle comes equipped with the equivalent norm $\Vert\cdot\Vert_\perp$, cf.~\bref{eq:normalestimate}. Define the map $\proj:NM\rightarrow M \times \mR$ by
\begin{equation*}
\proj(q,\xi)=(q,\Vert\xi\Vert_{\perp}).
\end{equation*}
Define $S=\proj^{-1}(N_\nu\times\{\rho\})$. This is a sphere sub-bundle of radius $\rho$ in the normal bundle over $N_\nu$. Recall that the inclusion of $M$ as the zero section in $NM$ is denoted by $\incl:M\rightarrow NM$. Set 
$$Z=\proj(\incl(U)\cup f(U))=U\times\{0\}\cup \proj(f(U)).$$
The sets are depicted in Figure~\ref{fig:projectednormal1}. Because $W$ $(k)$-links $N_\nu$ in $M$, the set $Z$ $(k+1)$-links $\proj(S)=N_\nu\times\{\rho\}$ in $M\times \mR$, as we prove below.

\begin{lemma}
\label{lem:step1}
The inclusion
$
H_{k+1}(Z)\rightarrow H_{k+1}(M\times \mR\setminus N_\nu\times\{\rho\})
$
is non-trivial. 
\end{lemma}

\begin{proof}

Recall that $W=|\cw|$ and $\cw=\partial \cu$ with $\cu\in C_{k+1}(M)$ a $(k+1)$-cycle. Define the cycle $\cx\in C_{k+1}(Z)$ by
\begin{equation*} 
\cx=\proj_{k+1}\incl_{k+1}(\cu)-\proj_{k+1}f_{k+1}(\cu),
\end{equation*}
This cycle is closed, because
\begin{align*}
\partial \cx=&\proj_{k}\incl_{k}(\partial \cu)-\proj_{k}f_{k}(\partial \cu)\\
=&\proj_{k}\incl_{k}(\cw)-\proj_{k}f_{k}(\cw)=0.
\end{align*}
In the last step we used that $f\bigr|_W=\incl\bigr|_W$. Hence $\hx\in H_{k+1}(Z)$. We show that this class is mapped to a non-trivial element in $H_{k+1}(M\times \mR \setminus N_\nu\times \{\rho\})$

For technical reasons we need to modify $Z$ and $N_\nu\times \{\rho\}$. Define the set $\tilde Z$ 
by
\begin{equation*}
 \tilde Z=U\times\{0\}\cup W\times[0,\rho]\cup T_{\rho}(\proj f(U)),
\end{equation*} 
where $T_{\rho}:M\times \mR\rightarrow M\times \mR$ is the translation over $\rho$ in the $\mR$ direction, i.e. $T_\rho(q,r)=(q,r+\rho)$. Denote by $I_\rho$ the interval $(\frac{\rho}{3},\frac{2\rho}{3})$. There exists a homotopy $\ho_t:M\times \mR\rightarrow M \times \mR$, with the following properties:
\begin{enumerate}
\item[(i)] $\ho_0=\id,$
\item[(ii)] $\ho_t(\tilde Z)\cap \ho_t(N_\nu\times I_\rho)=\emptyset,\quad\text{for all}\quad t,$
\item[(iii)] $\ho_1(\tilde Z)=Z,$
\item[(iv)] $\ho_1(N_\nu\times I_\rho)=N_\nu\times\{\rho\}.$
\end{enumerate}
These properties ensure that $Z$ $(k+1)$-links $N_\nu\times\{\rho\}$ if and only if $\tilde Z$ $(k+1)$-links $N_\nu\times I_\rho$. 
Define $[\cxt]=(\ho_1)^{-1}_{k+1}\hx\in H_{k+1}(\tilde Z)$. This is well defined because $(m_1)_{k+1}$ is an isomorphism. We will reason that this class includes non-trivially in $H_{k+1}(M\times \mR \setminus N_\nu\times I_\rho)$. For this we apply Mayer-Vietoris to the triad $(\tilde Z, U_1,U_2)$, with 
\begin{equation*}
\begin{split}
U_1=&U\times \{0\}\cup W\times [0,\frac{2\rho}{3}),\\
U_2=& W\times (\frac{\rho}{3},\rho]\cup T_{\rho}\,\proj f(U).\\
\end{split}
\end{equation*}
Note that $U_1\cap U_2= W\times I_\rho$. From the Mayer-Vietoris sequence for the triad we get the boundary map
\begin{equation*}
\xymatrix{H_{k+1}(\tilde Z)\ar[r]^-{\delta}&H_k(W\times I_\rho)}.
\end{equation*}
By definition of the boundary map $\delta$ in the Mayer-Vietoris sequence, we have that $\delta[\cxt]=(m_1)^{-1}_{k}\,\proj_{k}\,\incl_{k}\hw$. Now we consider a second Mayer-Vietoris sequence, the Mayer-Vietoris sequence of the triad
\begin{equation*}
\left(M\times \mR \setminus N_\nu\times I_\rho,M\times \mR_{>\frac{\rho}{3}} \setminus N_\nu\times I_\rho,M\times \mR_{<\frac{2\rho}{3}}\setminus N_\nu\times I_\rho\right).
\end{equation*}
By naturality of Mayer-Vietoris sequences, the following diagram commutes
\begin{equation*}
\xymatrix{
H_{k+1}(\tilde Z)\ar[d]_-{i_{k+1}}\ar[r]^-{\delta}&H_{k}(W\times I_\rho)\ar[d]^-{i_k}\\
H_{k+1}(M\times \mR\setminus N_\nu\times I_\rho)\ar[r]^-{\delta}&H_{k}(M\times I_\rho\setminus N_\nu\times I_\rho).\\
}
\end{equation*}
We argued that $\delta[\cxt]=(m_1)^{-1}_{k} \,\proj_{k}\,\incl_{k}\hw$. We have that $i_k(m_1)^{-1}_k\,\proj_k\,\incl_k\hw\not=0$ by assumption. By the commutativity of the above diagram we conclude that $i_{k+1}[\cxt]\not=0$. Thus $\tilde Z$ $(k+1)$-links $N_\nu\times I_\rho$ in $M\times\mR$, which implies that $Z$ $(k+1)$-links $N_\nu\times \{\rho\}$ in $M\times \mR$.
\end{proof}

The previous lemma lifted the link in the base manifold to a link in $M\times \mR$. We now lift this link to the full normal bundle.
\begin{lemma}
\label{lem:step2}
The inclusion
$H_{k+1}(\incl(U)\cup f(U))\rightarrow H_{k+1}(NM\setminus S)$ is non-trivial.
\end{lemma}
\begin{proof}
The following diagram commutes
\begin{equation*}
\xymatrix{
H_{k+1}(\incl(U)\cup f(U))\ar[r]^-{\proj_{k+1}}\ar[d]_{i_{k+1}}&H_{k+1}(Z)\ar[d]^{i_{k+1}}\\
H_{k+1}(NM\setminus S)\ar[r]^-{\proj_{k+1}}&H_{k+1}(M\times\mR\setminus N_\nu\times\{0\}).
}
\end{equation*}
Define $\hyp=\incl_{k+1}\hu - f_{k+1}\hu$. Recall that $\hx=\proj_{k+1}\hyp$ includes non-trivially in $H_{k+1}(M\times\mR\setminus N_\nu\times\{\rho\})$ by the construction in lemma \ref{lem:step1}. By the commutativity of the above diagram $\pi_{k+1}i_{k+1}\hyp=i_{k+1}\pi_{k+1}\hyp\not=0$. Thus $i_*\hyp\not=0$. The inclusion $H_{k+1}(\incl(U)\cup f(U))\rightarrow H_{k+1}(NM\setminus S)$ is non-trivial. 
\end{proof}

The domain of $\A$ is not the free loop space $\Lambda M$, but $\Lambda M\times \mR$. The extra parameter keeps track of the period of the candidate periodic solutions. Thus we need once more to lift the link to a bigger space. In this process we also globalize the link, moving it from the normal bundle to the full free loop space. Recall that we write $E=\Lambda M\times \mR$. The subsets $A^t=A_I\cup A_{II}\cup A^t_{III}$ are defined by
\begin{equation*}
\begin{split}
A_{I}&=\phi(\incl(U))\times\{\sigma_1\}\\
A_{II}&=\phi(\incl(W))\times[\sigma_1,\sigma_2]\\
A^t_{III}&=h_t(U)\times\{\sigma_2\}\\
\end{split}
\end{equation*}
The constants $\sigma_1<\sigma_2$ will be specified in Proposition \ref{prop:estimates}. Finally we define the sets $A, B\subset E$ by
\begin{equation}
\label{eq:AB}
A=A^1\qquad\text{and}\qquad B=\phi(S)\times\mR.
\end{equation}
Figure~\ref{fig:belly} depicts the sets $A$ and $B$.

\begin{lemma}
\label{lem:link}
The inclusion $H_{k+1}(A)\rightarrow H_{k+1}(E\setminus B)$ is non-trivial.
\end{lemma}
\begin{proof}
By Lemma \ref{lem:step2} the morphism induced by the inclusion $H_{k+1}(\incl(U)\cup f(U))\rightarrow H_{k+1}(NM\setminus S)$ is non-trivial. By applying the diffeomorphism $\phi$, we see therefore that 
\begin{equation*}
H_{k+1}(\phi(\incl(U))\cup \phi( f(U)))\rightarrow H_{k+1}(\tub\setminus \phi (S)),
\end{equation*}
is non-trivial. 
Because the base manifold (seen as the constant loops) is a retract (but not necessarily a deformation retract) of the loopspace we have the following relation:
\begin{equation}
\label{eq:splithomology}
H_*(\Lambda M)\cong H_*(M)\oplus H_*(\Lambda M, M).
\end{equation}
We assume $H_{k+2}(\Lambda M)=0$, thus $H_{k+2}(\Lambda M,M)=0$. The tubular neighborhood $\cV$ deformation retracts to $M$, hence we have $H_{k+2}(\Lambda M,\cV)\cong H_{k+2}(\Lambda M,M)$. Since $\phi(S)$ is closed and contained in the interior of $\cV$, we can excise $\phi(S)$. This gives an isomorphism $H_{k+2}(\Lambda M\setminus  \phi(S),\cV\setminus  \phi(S))\cong H_{k+2}(\Lambda M,\cV)\cong 0 $. The long exact sequence of the pair $(\Lambda M\setminus \phi(S),\cV \setminus  \phi(S))$ now gives that $H_{k+1}(\tub\setminus \phi(S))\rightarrow H_{k+1}(\Lambda M\setminus \phi(S))$ is injective.
It follows that

$H_{k+1}(\phi(\incl(U)\cup f(U)))\rightarrow H_{k+1}(\Lambda M\setminus \phi(S)),$

is non-trivial. Let $\pi_1:\Lambda M \times \mR\rightarrow \Lambda M$ be the projection to the first factor. Because of the choice of $\rho$, cf. Equation~\bref{eq:rho} the set $\pi_1(A^t)$ never intersects $\pi_1(B)$. By the construction of the sets $A^t$ and $B$, the map $\pi_1$ induces a homotopy equivalence between $A^t$ and $\pi_1(A^t)$ and between $E\setminus B$ and $\Lambda M\setminus \pi_1(B)$, so that the diagram
\begin{equation*}
\xymatrix{
H_{k+1}(A^t)\ar[d]_{(\pi_1)_{k+1}}\ar[r]&H_{k+1}(E\setminus B)\ar[d]^{(\pi_1)_{k+1}}\\
H_{k+1}(\pi_1(A^t))\ar[r]&H_{k+1}(\Lambda M\setminus \pi_1(B)),\\
}
\end{equation*}
commutes. We see that $H_{k+1}(A^t)\rightarrow H_{k+1}(E\setminus B)$ is non-trivial if and only if $H_{k+1}(\pi_1(A^t))\rightarrow H_{k+1}(\Lambda M \setminus \pi_1(B))$ is non-trivial. For all $t\in[0,1]$ the induced maps are the same, because of homotopy invariance. For $t=0$ we have that $\pi_1(A^0)=\phi(\incl(U)\cup f(U))$, and $\Lambda M\setminus \pi_1(B)=\Lambda M\setminus \phi(S)$. We conclude that $A$ $(k+1)$-links $B$ in $\Lambda M$.
\end{proof}
\begin{figure}
\def\svgwidth{.6\textwidth}
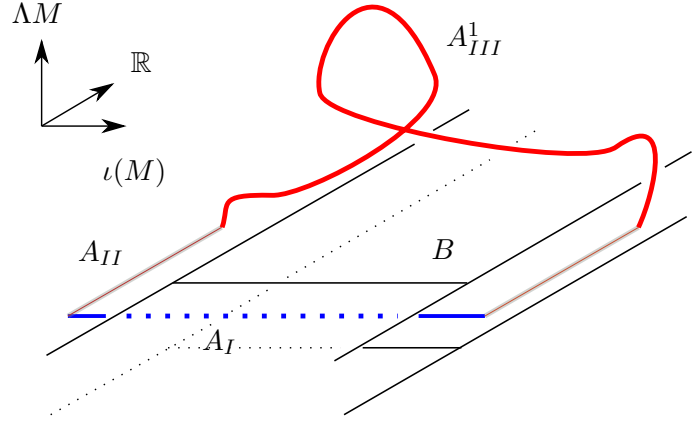
\caption{A sketch of the linking sets $A=A_I\cup A_{II}\cup A^1_{III}$ and $B$. For graphical reasons the base manifold $M$ and the loop space $\Lambda M$ are depicted as one-dimensional.}

\label{fig:belly}
\end{figure}
\section{Estimates}
\label{sec:estimates}
We need to estimate $\A$ on the sets $A,B\subset E$, defined in Equation~\bref{eq:AB}.

\begin{proposition}
\label{prop:estimates}
If $\nu$ and $\rho$ are sufficiently small, then there exist constants $\sigma_1<\sigma_2$ and $0<a<b$,  such that\begin{equation}
\A\bigr|_A\leq a\qquad\text{and}\qquad \A\bigr|_B>b.
\end{equation}
\end{proposition}
\begin{proof}
We first estimate $\A$ on $B=\phi(S)\times \mR$. Let $(c_1,\tau)\in \phi(S)\times\mR$. Then $c_1=\phi(\xi)=\exp_{c_0}(\xi)$ where $\xi$ is a vector field along a constant loop $c_0$ at $q\in N_\nu$, for which $\Vert\xi\Vert_{\perp}=\Vert\triangledown\xi\Vert_{L^2}=\rho$. From the Gauss lemma, and the following estimate, cf. Equations~\bref{eq:inequality} and~\bref{eq:normalestimate},
$$
\Vert{\xi}\Vert_{C^0}\leq \sqrt{2}\Vert\xi\Vert_{H^1}\leq 2\Vert\xi\Vert_\perp
$$
we see that $\sup_{s\in\mS^1}d_M(c_0(s),c_1(s))\leq 2\rho$. Recall that we assumed $\rho\leq \frac{\rho_\nu}{2}$. Hence for all $s\in\mS^1$, we have $c_1(s)\in B_{\rho_\nu}(q)$ and therefore $V(c_1(s))\leq-\frac{\nu}{2}$, by Lemma \ref{lem:modifyN}. We use this to estimate the second term of
\begin{equation}
\A(c_1,\tau)=\frac{e^{-\tau}}{2}\int_0^1|c'_1(s)|^2ds-e^\tau\int_0^1V(c_1(s))ds.\\
\end{equation}
Let us now concentrate on the first term. We construct the geodesic from $c_0$ to $c_1$ in the loop space, namely $c_t(s)=\exp_{c_0}(t\xi(s)).$ This can also be seen as a singular surface in $M$, cf.~\cite{Klingenberg_Riemannian}. Now we apply Taylor's formula with remainder to $t\mapsto \cE(c_t)$. There exists a $0\leq \tilde t\leq 1$ such that
\begin{equation}
\label{eq:Taylor}
\cE(c_1)=\cE(c_0)+\frac{d}{dt}\cE(c_t)\Bigr|_{t=0}+\frac{1}{2}\frac{d^2}{dt^2}\cE(c_t)\Bigr|_{t=0}+\frac{1}{6}\frac{d^3}{dt^3}\cE(c_t)\Bigr|_{t=\tilde t}.
\end{equation}
The first term $\cE(c_0)=0$, since $c_0$ is a constant loop. Because $t\mapsto c_t$ is a geodesic $\frac{d}{dt}\bigr|_{t=0}\cE(c_t)=0$. The second order neighborhood of a closed geodesic is well studied~\cite[Lemma 2.5.1]{Klingenberg_Riemannian}. We see that $c_0$ is a (constant) closed geodesic, therefore
$$
\frac{d^2}{dt^2}\Bigr|_{t=0}\cE(c_t)=D^2\cE(c_0)(\xi,\xi)=\Vert\xi \Vert_{\perp}^2=\rho^2.
$$ 
The curvature term in the second variation vanishes at $t=0$, because $c_0$ is a constant loop. The third derivative of the energy functional can be bounded in terms of the curvature tensor and its first covariant derivative times a third power of $\Vert\xi\Vert_\perp$. By the assumption of bounded geometry, we can therefore uniformly bound $\cE(c_1)$.
The main point is that for $\rho$ sufficiently small, $\cE(c_1)\geq C \rho^2$, for some constant $C>0$. We now can estimate $\A$ on $B$.
\begin{equation}
\A(c_1,\tau)\geq \frac{e^{-\tau}}{2}C\rho^2+\frac{e^\tau}{2}\nu\geq \sqrt{C\,\nu}\rho
\end{equation}
Set $b=\sqrt{C\,\nu}\rho$, then $\A|_B>b$. It remains to estimate $\A$ on the set $A=A_I\cup A_{II}\cup A^1_{III}$. Let $(c,\sigma_1)\in A_I=\phi(\incl(U))\times \{\sigma_1\}$. Recall that $U$ is compact, hence $V_{\max}=\sup_{q\in U}-V(q)<\infty$. Because $c$ is a constant loop, we find
\begin{equation}
\A(c,\sigma_1)=-e^{\sigma_1}\int_0^1 V(c(s))ds\leq e^{\sigma_1} V_{\max}.
\end{equation}
By choosing $\sigma_1\leq \log(\frac{b}{2V_{\max}})$ we get $\A\bigr|_{A_I}\leq b/2$. On $A_{II}=\phi(W)\times[\sigma_1,\sigma_2]$ all the loops are constants as well, moreover their image is contained in $W$. The potential is positive on $W$ hence $\A\bigr|_{A_{II}}<0<\frac{b}{2}$. It remains to estimate $\A$ on $A^1_{III}=h_1(U)\times \{\sigma_2\}$. Recall that we constructed $h$ in such a way that for any $q\in U$ we have $\int_0^1V(h_1(q)(s))ds>0$. This gives
\begin{equation}
\begin{split}
\A(c,\sigma_2)=&\frac{e^{-\sigma_2}}{2}\int_0^1|h_1(q)'(s)|^2ds-e^{\sigma_2}\int_0^1V(h_1(q)(s))ds\\
\leq&\frac{e^{-\sigma_2}}{2}\int_0^1|h_1(q)'(s)|^2ds.
\end{split}
\end{equation}
Because $h$ is continuous and $U$ is compact, $\cE_{\max}=\sup_{q\in U}\cE(h_1(q))<\infty$. And therefore
\begin{equation}
\A(c,\sigma_2)\leq \frac{e^{-\sigma_2}}{2}\cE_{\max}.
\end{equation}
By setting $\sigma_2>\max(\log(\frac{\cE_{\max}}{b}),\sigma_1)$ we get $\A\bigr|_{A^1_{III}}\leq \frac{b}{2}$. Now set $a=b/2$, and we see that $\A|_A<a<b$.
\end{proof}

\section{Proof of the main theorem}
\label{sec:proof}

\begin{proof}[Proof of Theorem~\ref{thm:main2}]

From the assumptions $H_{k+1}(\Sigma)\not=0$ and $H_{k+1}(\Lambda M)=H_{k+2}(\Lambda M)=0$, we are able to construct linking sets $A$ and $B$ in the loop space,  cf.~Lemma~\ref{lem:link}. We estimate $\A$ on $A$ and $B$ in Proposition~\ref{prop:estimates}. We now use Lemma 13 of~\cite{bpv}, with the minor caveat that the proof uses the fact that $H_{k+1}(\Lambda M\times\mR)\cong 0$ after formula (17) of this paper, which is automatically true for $M=\mR^{2n}$, but which we assume a priori here.   
\end{proof}

\bibliographystyle{abbrv}
\bibliography{weinstein}

\begin{thebibliography}{10}

\bibitem{benci}
V.~Benci.
\newblock {Closed geodesics for the Jacobi metric and periodic solutions of
  prescribed energy of natural Hamiltonian systems}.
\newblock {\em Annales de l'Institut Henri Poincar{\'e}. Analyse Non
  Lin{\'e}aire}, 1(5):401--412, 1984.

\bibitem{bolotin}
S.~V. Bolotin.
\newblock {Libration motions of natural dynamical systems}.
\newblock {\em Vestnik Moskovskogo Universiteta. Seriya I. Matematika,
  Mekhanika}, 1(6):72--77, 1978.

\bibitem{floerhoferwysocki}
A.~Floer, H.~Hofer, and K.~Wysocki.
\newblock {Applications of symplectic homology. I}.
\newblock {\em Mathematische Zeitschrift}, 217(4):577--606, 1994.

\bibitem{Gluck:1983ut}
H.~Gluck and W.~Ziller.
\newblock {Existence of periodic motions of conservative systems}.
\newblock In {\em Seminar on minimal submanifolds}, pages 65--98. Princeton
  Univ. Press, Princeton, NJ, 1983.

\bibitem{Hirsch}
M.~W. Hirsch.
\newblock {\em {Differential topology}}, volume~33 of {\em Graduate Texts in
  Mathematics}.
\newblock Springer-Verlag, New York, 1994.

\bibitem{hoferviterbo}
H.~Hofer and C.~Viterbo.
\newblock {The Weinstein conjecture in cotangent bundles and related results}.
\newblock {\em Annali della Scuola Normale Superiore di Pisa. Classe di
  Scienze. Serie IV}, 15(3):411--445 (1989), 1988.

\bibitem{Klingenberg}
W.~Klingenberg.
\newblock {\em {Lectures on closed geodesics}}.
\newblock Springer-Verlag, Berlin-New York, 1978.

\bibitem{Klingenberg_Riemannian}
W.~P.~A. Klingenberg.
\newblock {\em {Riemannian geometry}}, volume~1 of {\em de Gruyter Studies in
  Mathematics}.
\newblock Walter de Gruyter and Co., Berlin, second edition, 1995.

\bibitem{munkres}
J.~R. Munkres.
\newblock {\em {Topology: a first course}}.
\newblock Prentice-Hall, Inc., Englewood Cliffs, N.J., 1975.

\bibitem{palaishilbert}
R.~S. Palais.
\newblock {Morse theory on Hilbert manifolds}.
\newblock {\em Topology. An International Journal of Mathematics}, 2:299--340,
  1963.

\bibitem{rabinowitzstar}
P.~H. Rabinowitz.
\newblock {Periodic solutions of Hamiltonian systems}.
\newblock {\em Communications on Pure and Applied Mathematics}, 31(2):157--184,
  1978.

\bibitem{rabinowitzenergy}
P.~H. Rabinowitz.
\newblock {Periodic solutions of a Hamiltonian system on a prescribed energy
  surface}.
\newblock {\em Journal of Differential Equations}, 33(3):336--352, 1979.

\bibitem{Rot:ww}
T.~O. Rot.
\newblock {\em {PhD Thesis}}.
\newblock To appear.

\bibitem{suhrzehmisch}
S.~Suhr and K.~Zehmisch.
\newblock {Linking and closed orbits}.
\newblock {\em arXiv.org}, May 2013.

\bibitem{bpv}
J.~B. van~den Berg, F.~Pasquotto, and R.~C. Vandervorst.
\newblock {Closed characteristics on non-compact hypersurfaces in
  $\mathbb{R}^{2n}$}.
\newblock {\em Mathematische Annalen}, 343(2):247--284, 2009.

\bibitem{viterboweinstein}
C.~Viterbo.
\newblock {A proof of Weinstein's conjecture in $\mathbb{R}^{2n}$}.
\newblock {\em Annales de l'Institut Henri Poincar{\'e}. Analyse Non
  Lin{\'e}aire}, 4(4):337--356, 1987.

\bibitem{viterboexact}
C.~Viterbo.
\newblock {Exact Lagrange submanifolds, periodic orbits and the cohomology of
  free loop spaces}.
\newblock {\em J. Differential Geom.}, 47(3):420--468, 1997.

\bibitem{viterbofunctors}
C.~Viterbo.
\newblock {Functors and computations in Floer homology with applications. I}.
\newblock {\em Geometric and Functional Analysis}, 9(5):985--1033, 1999.

\bibitem{weinsteinconvex}
A.~Weinstein.
\newblock {Periodic orbits for convex Hamiltonian systems}.
\newblock {\em Annals of Mathematics. Second Series}, 108(3):507--518, 1978.

\bibitem{weinsteinhypothesis}
A.~Weinstein.
\newblock {On the hypotheses of Rabinowitz' periodic orbit theorems}.
\newblock {\em Journal of Differential Equations}, 33(3):353--358, 1979.

\end{thebibliography}

\end{sloppypar}
\end{document}